\documentclass[12pt]{article}
\usepackage[top=1in,bottom=1in,left=1 in,right=.8in,footskip=0.4in]{geometry}
\usepackage{authblk,orcidlink}
\usepackage{amsfonts,amsmath,amssymb,amsthm}

\newtheorem{theorem}{Theorem}[section]
\newtheorem{lemma}[theorem]{Lemma}
\newtheorem{proposition}[theorem]{Proposition}

\theoremstyle{definition}
\newtheorem{definition}[theorem]{Definition}
\theoremstyle{remark}

\usepackage{mathrsfs}   
\usepackage{bm}         

\usepackage{cite}

\usepackage{graphicx}

\usepackage{hyperref}
\hypersetup{colorlinks=true,linkcolor=blue,citecolor=blue,urlcolor=blue}
\usepackage{microtype}

\usepackage{enumerate}

\newcommand{\rk}{\operatorname{rk}}
\newcommand{\crit}{\operatorname{crit}}
\newcommand{\up}{\operatorname{U}}
\newcommand{\dn}{\operatorname{D}}

\newcommand{\M}{\mathcal{M}}

\title{$2$-colourability of the maximum ranked elements of a combinatorially sphere-like ranked poset}

\author[a]{Anupam Mondal~\orcidlink{0000-0002-6547-4835}\thanks{\texttt{anupam.mondal@tcgcrest.org }}}

\author[b]{Sajal Mukherjee~\orcidlink{0009-0004-6959-4334}\thanks{\texttt{sajal.mukherjee@tcgcrest.org}}}

\author[c]{Pritam Chandra Pramanik~\orcidlink{0009-0000-8332-3347}\thanks{\texttt{pritam.pramanik.80@tcgcrest.org}}}

\affil[a,b,c]{\small Institute for Advancing Intelligence (IAI), TCG CREST, Kolkata--700091, West Bengal, India}
\affil[a,b]{\small Mathematical \& Information Science, Academy of Scientific and Innovative Research (AcSIR), Ghaziabad--201002, Uttar Pradesh, India}

\date{}

\begin{document}
	\maketitle
	\begin{abstract}
	We obtain a higher dimensional analogue of a classical theorem which states that a polygonally cellulated $2$-sphere in $\mathbb{R}^3$, such that each vertex has even degree, is $2$-face-colourable. In order to formulate our result, we introduce the notion of combinatorially sphere-like ranked posets, which are ranked posets that generalise combinatorial spheres. We prove that, in a combinatorially sphere-like ranked poset $S$ of rank $k$, if each element of rank $(k-2)$ is covered by an even number of elements, then the maximum ranked elements of $S$ admit a proper $2$-colouring, i.e., any two adjacent maximum ranked elements have different colours.
	\end{abstract}
	\noindent \textbf{Keywords:} ranked poset, $2$-colourability, planar graph, Eulerian graph, polyhedral graph, cellulated sphere, acyclic matching
	
	\noindent \textbf{MSC 2020:} 06A07, 05C15, 52B05
	\section{Introduction}
	Our goal is to obtain a higher dimensional analogue of the following theorem, which is of interest in combinatorics and geometry.
	\begin{theorem}\label{moti} A polygonally cellulated $2$-sphere in $\mathbb{R}^3$, such that each vertex has even degree, is $2$-face-colourable.
	\end{theorem}
	Since, polygonally cellulated $2$-spheres can be realised as polyhedral graphs, which are $3$-connected planar graphs, the theorem above follows from a classic and elegant result from graph theory. 
	
	\begin{theorem}\label{planegraph}
	     A planar dual of an Eulerian  planar graph is bipartite, in other words any planar embedding of an Eulerian planar graph is $2$-face-colourable.  
	\end{theorem} 
	A natural generalisation of this polygonally cellulated $2$-sphere is \emph{combinatorial $d$-sphere} (i.e., a $d$-dimensional \emph{psudomanifold} which \emph{collapses}~\cite[Chapter 1 \& 2]{cohen} to a point after the deletion of a specific $d$-dimensional face). However, instead of combinatorial spheres, we work with a more generalised poset theoretic setup. We introduce the notion of \emph{combinatorially sphere-like ranked poset} or, simply \emph{sphere-like poset}.
	In order to formally define sphere-like posets and to state our main result, first we define the following terminologies.

	The \emph{directed Hasse diagram} of a ranked poset $S$ is the directed acyclic graph on $S$, where an edge is directed from $x$ to $y$, if and only if $x$ covers $y$. An \emph{acyclic matching} on a ranked poset $S$ is a matching in the directed Hasse diagram of $S$, such that the graph, obtained after altering the direction of the edges in the matching, remains acyclic. For an acyclic matching $\M$ on $S$, an element $x\in S$ is said to be $\M$-critical if $x$ is not matched by $\M$. It is noteworthy that the notion of acyclic matching comes from discrete Morse theory\cite{chari}, and they are associated with the notion of \emph{collapse}~\cite{forman2002,forman1998,chari}.   
	 \begin{definition}[Sphere-like poset]
	 	A ranked poset $S$ of rank $k$ (where $k\geq 2$) is said to be \emph{sphere-like} if the following holds.
	 	\begin{enumerate}[(i)]
	 		\item Each element of rank $(k-1)$ is covered by exactly two elements.
	 		\item For each pair $(y,z)$ with $\rk(y)=k$ and $\rk(z)=k-2$, there are an even number of elements $x$ of rank $(k-1)$ such that $z \lessdot  x\lessdot y$.
	 		\item There exists an acyclic matching $\M$ on $S$, such that no element of rank $(k-1)$ is $\M$-critical.
	 	\end{enumerate}
	 \end{definition}
	 
	In this note, we provide a self-contained and purely combinatorial proof of the following result about the $2$-colourability of ``faces" (elements of the maximum rank) of sphere-like posets.
		\begin{theorem}\label{tmain}
			Let $S$ be a sphere-like poset of rank $k$. If for every $z\in S$ with $\rk(z)=k-2$, the number of elements which covers $z$ is even, then the maximum ranked elements of $S$ admit a proper $2$-colouring (i.e., any two maximum ranked elements covering a `common' element have different colours).
		\end{theorem}
    Next, we justify that any planar embedding of a planar graph may be realised as a sphere-like poset of rank $2$, and thus Theorem~\ref{planegraph} follows from Theorem~\ref{tmain}.
	\section{Preliminaries}
	
	Let $S$ a partially ordered set (or, poset) with the partial order `$\leq$' on $S$. For $x,y \in S$, we say \emph{$y$ covers $x$} (often written as $x\lessdot y$) if (i) $x<y$, and (ii) there is no element $z$ such that $x<z<y$. A \emph{ranked poset} is a poset $S$ equipped with a rank function $\rk:S \rightarrow \mathbb{N}\cup \{0\}$, such that `$\rk$' satisfies the following two properties : (i) the rank function is compatible with the partial order, that is, if $x < y$ then $\rk(x) < \rk(y)$, and
	(ii) the rank is consistent with the covering relation, that is, if $x \lessdot y$ then $\rk(y) = \rk(x)+1$.  The \emph{rank of a poset} $S$ is defined as, $\rk(S):= \max\{\rk(x): x \in S\}$. For the sake of convenience, whenever $x$ does not cover any elements, we assume that $\rk(x)=0$. For any $x\in S$, let $\Delta(x):= \{z\in S: z \lessdot x\}$, and $\nabla(x):= \{y\in S: x \lessdot y\}$.
	For any $x_1,x_2\in S$, such that $\rk(x_1)=\rk(x_2)$, we say $x_1$ and $x_2$ are \emph{adjacent} if $\Delta(x_1)\cap \Delta(x_2)\neq \emptyset$.
	
    Let, for a graph $G$, $V(G)$ and $E(G)$ be the \emph{vertex set} and \emph{edge set} of $G$, respectively. A graph $G$ can be realised as a poset, $\mathcal{F}(G)=(V(G)\sqcup E(G),\leq)$, where the partial order `$\leq$' is defined as, for any $v\in V(G)$ and $e\in E(G)$, $v \leq e$ if and only if $v$ is an endpoint $e$. A graph $G$ is \emph{planar} if it admits a planar embedding. For a planar embedding of a planar graph $G$, along with vertices and edges, there is a natural notion of \emph{faces}. The \emph{face poset} of a planar graph $G$ (with respect to a chosen embedding) is a poset on the set $V(G)\sqcup E(G)\sqcup F(G)$, where $F(G)$ is the set of faces of $G$, and the partial order is given as follows. For any $v\in V(G)$ and $e\in E(G)$, $v \leq e$ if and only if $v$ is an endpoint of $e$ and for any $e\in E(G)$ and $f\in F(G)$, $e \leq f$ if and only if $e$ lies in the boundary of the face $f$. We refer to \cite{diestel} for the notions related to graph theory.
    
	Let $S$ be a ranked poset. We note that, a \emph{matching} in the directed Hasse diagram of $S$   is a collection $\M$ 
	of ordered pairs of elements in $S$, such that (i) if $(x,y)\in \M$, then $y$ covers $x$, and (ii) each element of $S$ is in at most one pair of $\M$.
	An \emph{$\M$-path} is an alternating sequence of elements,
	$P:y_0,x_1, y_1, \ldots, x_r,
	y_r,$
	such that, for each $i \in \{1,\ldots,r\}$, $(x_i,y_i) \in
	\M$, and $y_{i-1} (\neq y_i)$ covers $x_i$. For any $(x,y)\in \M$, we sometimes represent the pair $(x,y)$ as $x \rightarrowtail y$. With this notation a representation of $P$ goes as follows.
	\[P:y_0\gtrdot x_1 \rightarrowtail y_1 \gtrdot \cdots \gtrdot  x_r \rightarrowtail
	y_r.\]
	The elements $y_0$ and $y_r$ are called the \emph{initial element} (denoted by $\operatorname{init}(P)$) and the \emph{terminal element} (denoted by $\operatorname{term}(P)$) of the $\M$-path $P$, respectively. Such an $\M$-path $P$ is said to be \emph{non-trivial closed} if
	$r > 0$ and $x_{r} = x_0$. We observe that, $\M$ is an \emph{acyclic matching} on $S$ if there is no non-trivial closed $\M$-paths.
	Let $S$ be a ranked poset and $\M$ be an acyclic matching on $S$, then  an element $x\in S$ is said to be \emph{$\M$-critical} if  $x$ does not appear in any pair of $\M$. Let
	\begin{align*}
		\crit_q^{(\M)} &:=\{x\in S: x \text{ is $\M$-critical, } \rk(x)=q\},\\
		\up_q^{(\M)} &:=\{x\in S: (x,y)\in \M, \rk(x)=q)\}\text{, and}\\
		\dn_q^{(\M)} &:=\{x\in S: (z,x)\in \M, \rk(x)=q)\}.
	\end{align*} 
	
	 Note that, in a rooted tree $T$, a matching that pairs each edge of $T$  with its endpoint farthest from the root, is acyclic. Hence, we have the following proposition.
	\begin{proposition}\label{tree}
	    Let $T$ be a tree. Then there is an acyclic matching on $\mathcal{F}(T)$, such that all the edges of $T$ (i.e., elements of rank $1$ in $\mathcal{F}(T)$) are matched. 
	\end{proposition}

	 The notion of $2$-colourability of the maximum ranked elements of a sphere-like poset
	 is as follows.
	 \begin{definition}
	 	Let $S$ be a sphere-like poset of rank $k$, and $S_k=\{x\in S: \rk(x)=k\}$. Then, we say the maximum ranked elements of $S$ are \emph{$2$-colourable} (or simply, $S$ is $2$-colourable) if there exists a function $\psi: S_k\rightarrow \mathbb{Z}_2$, such that for any $x_1,x_2 \in S_k$, if $x_1$ and $x_2$ are adjacent, then $\psi(x_1)\neq \psi(x_2)$.
	 \end{definition}

	  In sphere-like poset $S$ of rank $k$, for $0\leq q\leq k$, let $S_q$ denotes the set of all elements of rank $q$ in $S$. Let, for any $q\in \{0,\ldots,k\}$, $C_q$ be the formal $\mathbb{Z}_2$-vector space with $S_q$ as a basis. We define linear maps $d_q: C_q \rightarrow C_{q-1}$ (by defining on the generating set $S_q$), such that, for $x\in S_q$,
	  \[d_q(x)= \sum_{z\in \Delta(x)}z.\]
	  Equivalently,
	 \[d_q(x)= \sum_{z\in S_{q-1}} \chi_{(x,z)}\cdot z, \text{ where}\]
	 \[ \chi_{(x,z)}=\begin{cases}
	     1, \text{ if } z\lessdot x,\\
	     0, \text{ otherwise.}
	 \end{cases}\]
	
	\begin{lemma}\label{poincare}
		Let  $S$ be sphere-like poset of rank $k$, then $d_{k-1}\circ d_{k}=0$.
	\end{lemma}
	\begin{proof}
		The proof of this lemma follows from the property that, in a sphere-like poset,  for each pair $(y,z)$ with $\rk(y)=k$ and $\rk(z)=k-2$, then there are an even number of elements $x$ of rank $(k-1)$, such that $z\lessdot x\lessdot y$.
	\end{proof}

	\section{Proof of the main result}
	 
	 	Let $S$ be a any ranked poset and $\M$ be any acyclic matching on $S$.  For any $x,y \in S$, such that $\rk(x)=\rk(y)$, let $\Gamma_{(x,y)}^{(\M)}$ be the collection of all $\M$-paths from $x$ to $y$. Let 
	 	\[\ell^{(\M)}_{(x,y)}:= |\Gamma_{(x,y)}^{(\M)}|  \pmod{2}, \text{ in other words}\]
	 \[\ell^{(\M)}_{(x,y)}= \begin{cases}
	 	1, \text{ if } |\Gamma_{(x,y)}^{(\M)}| \text{ is odd,}\\
	 	0, \text{ if } |\Gamma_{(x,y)}^{(\M)}| \text{ is even}.
	 \end{cases}\]
	 Then we have the following lemmas.
	 \begin{lemma}\label{ft}
	 	Let $S$ be a ranked poset and $\M$ be any acyclic matching on it. Then, for any $x\in \crit_q^{(\M)}$ and $z\in \up_{q-1}^{(\M)}$, we have  $\sum_{y\in \nabla(z)} \ell^{(\M)}_{(x,y)}=0$.
	 \end{lemma}
	 \begin{proof}
	 	Let $(z,y^\prime)\in \M$. We define a map $\phi:\Gamma^{(\M)}_{(x,y^\prime)} \rightarrow \cup_{y\in \nabla(z)\setminus \{y^\prime\}} \Gamma^{(\M)}_{(x,y)}$ as follows. For any $P\in  \Gamma^{(\M)}_{(x,y^\prime)} $, such that 
	 	\[P=x\gtrdot \cdots \rightarrowtail y \gtrdot z \rightarrowtail y^\prime, \]
	 	$\phi$ maps $P$ to a ``path" obtained after discarding the tail `$ z \rightarrowtail y^\prime$', that is
	 	\[\phi(P)= x \gtrdot \cdots \rightarrowtail y.\]
	 We observe that $\phi$ is a bijection, and hence we get
	 	\begin{align*}
	 		& |\Gamma^{(\M)}_{(x,y^\prime)}|=| \cup_{y\in \nabla(z)\setminus \{y^\prime\}}\Gamma^{(\M)}_{(x,y)}|\\
	 		\implies & |\cup_{y\in \nabla(z)} \Gamma_{(x,y)} ^{(\M)}|   \equiv 0 \pmod{2}\\
	 		\implies &  \sum_{y\in \nabla(z)}| \Gamma_{(x,y)}^{(\M)} | \equiv 0 \pmod{2}\\
	 		\implies & \sum_{y\in \nabla(z)} \ell^{(\M)}_{(x,y)}=0.
	 	\end{align*}
	 \end{proof}
    
     \begin{lemma}\label{dn}
     	Let $S$ be a sphere-like poset of rank $k$, and $\M$ be any acyclic matching on
     	 $S$. Let $ \dn_{k-1}^{(\M)}=\{x_1,\ldots,x_n\}$. For any $\sigma=\sum_{i=1}^n t_ix_i ~(\in C_{k-1})$, if $d_{k-1}(\sigma)=0$, then $\sigma=0$.
     \end{lemma}
     \begin{proof}
        Without loss of generality, let us assume $t_i=1$ for all $i\in \{1,\ldots,m\}$ and $t_i=0$ for all $i\in \{m+1,\ldots,n\}$, where $m\in  \{1,\ldots,n\}$. Then, we have
        \[\sigma=\sum^m_{i=1} x_i.\]  
        Now, 
        \begin{align} \label{eq1}
           \nonumber &d_{k-1}(\sigma)=0\\
        	\nonumber \implies & d_{k-1}(\sum^m_{i=1} x_i)=0\\
            \nonumber \implies & 	\sum^m_{i=1} d_{k-1}(x_i)=0\\
            \nonumber \implies & \sum_{i=1}^{m} \sum_{z\in S_{k-2}} \chi_{(x_i,z)} \cdot  z=0\\
            \nonumber \implies &  \sum_{z\in S_{k-2}} \left( \sum_{i=1}^{m}  \chi_{(x_i,z)} \right)  z=0\\
           \nonumber  \implies & \sum_{i=1}^{m}  \chi_{(x_i,z)} = 0\text{, for all } z \in S_{k-2}\\
            \implies & | \nabla(z) \cap \{x_1,\ldots,x_m\}| \text{ is even, for all }z \in S_{k-2}.
        \end{align} 
       Let, for all $i\in \{1,\ldots,m\}$, $(z_i,x_i)\in \M$. Now, let us assume that the following is a longest $\M$-path involving the elements $x_1,x_2\ldots,x_m$ and $z_1,z_2\ldots,z_m$.
       \[P:x_{i_0} \gtrdot z_{i_1} \rightarrowtail x_{i_1}\gtrdot \cdots \gtrdot z_{i_p} \rightarrowtail x_{i_p},\]
       where $i_j \in \{1,\ldots m\}$. Now, since $x_{i_0} \in \nabla(z_{i_0})$, from Equation~\eqref{eq1}, there exists some $x_r\neq x_{i_0}$, $r\in \{1,\ldots,m\}$, such that $x_r\in \nabla(z_{i_0})$. Now if $x_r\neq x_{i_j}$, for all $j\in\{0,\ldots,p\}$, then 
       \[x_r \gtrdot z_{i_0} \rightarrowtail x_{i_0} \gtrdot z_{i_1} \rightarrowtail x_{i_1}\gtrdot \cdots \gtrdot z_{i_p} \rightarrowtail x_{i_p}\]
       is also a $\M$-path, which contradicts the fact that $P$ is a longest $\M$-path. And if  $x_r= x_{i_j}$, for some $j\in \{1,\ldots,p\}$, 
       \[(x_{i_j}=)~x_r \gtrdot z_{i_0} \rightarrowtail x_{i_0} \gtrdot z_{i_1} \rightarrowtail x_{i_1}\gtrdot \cdots \gtrdot z_{i_j} \rightarrowtail x_{i_j},\]
       is a closed $\M$-path, which is again a contradiction. Therefore, $t_i=0$, for all $i\in \{1, \ldots, n\}$, and hence $\sigma=0$.

     \end{proof}
     \begin{lemma} \label{bd}
     	Let $S$ be a sphere-like ranked poset of rank $k$, then for any $\eta\in C_{k-1}$, such that $d_{k-1}(\eta)=0$, there exists $\xi\in C_k$ such that $d_k(\xi)=\eta$.
     \end{lemma}
     \begin{proof}
     	Let $\M$ be an acyclic matching on $S$, such that there exists no $\M$-critical element of rank $(k-1)$. Let $\eta=\sum_{i=1}^{n} s_i x_i$, where $\up_{k-1}^{(\M)}=\{x_1,\ldots,x_m\}$, $\dn_{k-1}^{(\M)}=\{x_{m+1},\ldots,x_{n}\}$, and $s_i\in \mathbb{Z}_2$, for all $i \in \{1,\ldots,n\}$.  Let $(x_j,y_j)\in \M$, for all $j\in \{1,\ldots,m\}$.	 Note that, for all $j\in\{1,\ldots,m\}$, $\M_j=\M\setminus\{(x_j,y_j)\}$ is also an acyclic matching on $S$.
     	For any $j\in \{1,\ldots,m\}$, we define $\overrightarrow{y_j}^{(\M_j)}$ ($\in C_k$), as follows.
     	\[ \overrightarrow{y_j}^{(\M_j)}:= \sum_{y\in S_k} \ell^{(\M_j)}_{(y_j,y)}\cdot y.\]
     
     	We show that, for $\xi= \sum_{j=1}^m s_j\overrightarrow{y_j}^{(\M_j)}$, $d_k(\xi)=\eta$. For any $j\in\{1,\ldots,m\}$, we have
     	\begin{align} \label{c}
     		\nonumber	d_k (\overrightarrow{y_j}^{(\M_j)}) =&d_k(\sum_{y\in S_{k}}\ell^{(\M_j)}_{(y_j,y)}\cdot y)\\
     		\nonumber	= & \sum_{y\in S_{k}}\ell^{(\M_j)}_{(y_j,y)}\cdot d_k(y)\\
     		\nonumber	= &\sum_{y\in S_{k}}\ell^{(\M_j)}_{(y_j,y)}\cdot \sum_{i=1}^n \chi_{(y,x_i)} \cdot x_i\\
     		\nonumber	= & \sum_{i=1}^n\sum_{y\in S_{k}}\ell^{(\M_j)}_{(y_j,y)}\cdot \chi_{(y,x_i)} \cdot x_i\\
     	    \nonumber		= & \sum_{i=1}^m\sum_{y\in S_{k}}\ell^{(\M_j)}_{(y_j,y)}\cdot \chi_{(y,x_i)} \cdot x_i+ \sum_{i=m+1}^n\sum_{y\in S_{k}}\ell^{(\M_j)}_{(y_j,y)}\cdot \chi_{(y,x_i)} \cdot x_i\\
     	    		=& \sum_{i=1}^m \sum_{y\in \nabla(x_i)}\ell^{(\M_j)}_{(y_j,y)} \cdot x_i	+ \sum_{i=m+1}^n \sum_{y\in \nabla(x_i)}\ell^{(\M_j)}_{(y_j,y)} \cdot x_i.	
     	\end{align}
   Note that, since $x_j\lessdot y_j$, there exists only one trajectory $P$, which is the trivial trajectory $P: y_j$, such that $\operatorname{init}(P)=y_j$ and $\operatorname{term}(P)\in \nabla(x_j)$, otherwise it violates the acyclicity of $\M$. In this case, $\operatorname{init}(P)=\operatorname{term}(P)=y_j$. Hence, $\sum_{y\in \nabla(x_j)}\ell^{(\M_j)}_{(y_j,y)}=1$. Now, since $y_j \in \crit_k^{(\M_j)}$, and, for any $x_i$, such that $i\in \{1,\ldots,m\}\setminus \{j\}$,  $x_i\in \up_{k-1}^{(\M_j)}$, from Lemma~\ref{ft}, we have $\sum_{y\in \nabla(x_i)}\ell^{(\M_j)}_{(y_j,y)}=0$. Hence Equation~\eqref{c} becomes
     	
     	\begin{align}  \label{a} 
     		d_k (\overrightarrow{y_j}^{(\M_j)})	= x_j +   \sum_{i=m+1}^n \sum_{y\in \nabla(x_i)}\ell^{(\M_j)}_{(y_j,y)} \cdot x_i	.
     	\end{align}
 
  Now,		
     	\begin{align}	\label{b}
     	\nonumber   d_k ( \sum_{j=1}^m s_j \overrightarrow{y_j}^{(\M_j)})=&\sum_{j=1}^m s_jd_k (\overrightarrow{y_j}^{(\M_j)}) \\
     		\nonumber	=&	\sum_{j=1}^m s_j \left(x_j +   \sum_{i=m+1}^n \sum_{y\in \nabla(x_i)}\ell^{(\M_j)}_{(y_j,y)} \cdot x_i	 \right) \text{ (from Equation~\eqref{a})}\\
     		\nonumber	=&	\sum_{j=1}^m s_j x_j +  \sum_{j=1}^m s_j \left( \sum_{i=m+1}^n \sum_{y\in \nabla(x_i)}\ell^{(\M_j)}_{(y_j,y)} \cdot x_i \right)\\
     		\nonumber	=&	\sum_{j=1}^m s_j x_j +   \sum_{i=m+1}^n \left( \sum_{j=1}^m  \sum_{y\in \nabla(x_i)}s_j \cdot \ell^{(\M_j)}_{(y_j,y)}  \right)\cdot x_i\\
     		\nonumber	=&	\sum_{i=1}^m s_i x_i +   \sum_{i=m+1}^n \left( \sum_{j=1}^m  \sum_{y\in \nabla(x_i)}s_j \cdot \ell^{(\M_j)}_{(y_j,y)}  \right)\cdot x_i\\
     		\nonumber	=&	\sum_{i=1}^m s_i x_i + \sum_{i=m+1}^n s_ix_i+  \sum_{i=m+1}^n \left( \sum_{j=1}^m  \sum_{y\in \nabla(x_i)}s_j \cdot \ell^{(\M_j)}_{(y_j,y)} +s_i \right)\cdot x_i\\
     		\nonumber	=&	\eta +  \sum_{i=m+1}^n t_i \cdot x_i \text{ (where, $t_i=\sum_{j=1}^m  \sum_{y\in \nabla(x_i)}s_j \cdot \ell^{(\M_j)}_{(y_j,y)} +s_i$)}\\
     		\implies  d_k ( \sum_{j=1}^m s_j \overrightarrow{y_j}^{(\M_j)})+\eta=& \sum_{i=m+1}^n t_i \cdot x_i.
     		\end{align}
     		
     		Now, since 
     		\begin{align*}
     		d_{k-1}(\sum_{i=m+1}^n t_i \cdot x_i)=&d_{k-1}( d_k ( \sum_{j=1}^m s_j \overrightarrow{y_j}^{(\M_j)})+\eta)\\
     		=& 	d_{k-1}( d_k ( \sum_{j=1}^m s_j \overrightarrow{y_j}^{(\M_j)}))+d_{k-1}(\eta)\\
     		=&	d_{k-1}( d_k ( \sum_{j=1}^m s_j \overrightarrow{y_j}^{(\M_j)}))\\
     		=& 0 \text{ (by Lemma~\ref{poincare})}, 
     		\end{align*}
     by Lemma~\ref{dn}, $\sum_{i=m+1}^n t_i \cdot x_i=0$. Hence from Equation~\eqref{b}, we get
    \begin{align*}
    	  &d_k ( \sum_{j=1}^m s_j \overrightarrow{y_j}^{(\M_j)})+\eta=0\\
    	  \implies &d_k ( \sum_{j=1}^m s_j \overrightarrow{y_j}^{(\M_j)})=\eta.
    \end{align*}	
     \end{proof}

    Now we prove our main result (Theorem~\ref{tmain}). 
	\begin{proof}[Proof of Theorem~\ref{tmain}]
	    Let $S_k=\{y_1,\ldots, y_m\}$ and $S_{k-1}=\{x_1,\ldots x_n\}$. Let $(d_k)_M$ be the matrix representation of the linear map $d_k: C_k \rightarrow C_{k-1}$, i.e., 
	    \[(d_k)_M=(a_{ij})_{n\times m}, \text{ where}\]
	    \[a_{ij}=\begin{cases}
	    	1, \text{ if } x_i \lessdot y_j,\\
	    	0, \text{otherwise.}
	    \end{cases}\]

	    Now, let us assume that $S$ is $2$-colourable and $\psi: S_k \rightarrow \mathbb{Z}_2$ be a function, such that $\psi$ gives different values on the adjacent elements. Note that, since any element of rank $(k-1)$ is covered by exactly two elements in $S$, there are exactly two $1$-s in each row of $(d_k)_M$ and rest of the entries are $0$ in that row. Therefore, for each $i\in \{1,\ldots,n\}$, we have
	   
	    \[a_{i1}\psi(y_1)+\cdots+ a_{im}\psi(y_m)=1.\]
	    Which implies, $(\psi(y_1), \psi(y_2), \ldots , \psi(y_m))^T$ is a solution of the equation $(d_n)_M X= B$, where $B= (\small\underbrace{1,1,\ldots, 1}_{n \text{ times}})^T$. Conversely, for any solution $(t_1,t_2,\ldots, t_m)^T$ of the equation $(d_n)_M X= B$, the function $\psi: S_k\rightarrow \mathbb{Z}_2$, defined by $\psi(y_j)=t_j$, for all $j\in \{1,\ldots,m\}$, gives different values for the adjacent elements. So, $S$ is $2$-colourable if and only if the equation $(d_n)_M X= B$ has a solution.
	    
	    Let $\eta= \sum_{i=1}^n x_i~(\in C_{k-1})$. Now, since any element of rank $(k-2)$ is covered by an even number of elements of rank $(k-1)$, we have $d_{k-1}(\eta)=0$. Therefore, by Lemma~\ref{bd}, there exists some $\xi \in C_k$, such that $d_k(\xi)=\eta$. Now, if $\xi=\sum_{j=1}^m s_jy_j$, $(s_1,s_2,\ldots,s_m)^T$ is a solution of the equation $(d_n)_M X= B$. Hence, $S$ is $2$-colourable.
	    \end{proof}
	    
	    Now we prove Theorem~\ref{planegraph} as an application of Theorem~\ref{tmain}. The only thing that needs a justification that the face poset of an Eulerian planar graph admits an acyclic matching that matches all the edges (i.e., the elements of rank $1$ of the face poset). 
	    \begin{proof}[Proof of Theorem~\ref{planegraph}]
	    Let $G$ be an Eulerian planar graph. Let $S^G$ be the face poset of $G$. We consider the planar dual $G^*$ of $G$, that is $G^*$ is a (multi-)graph, with $V(G^*)=F(G)$, and $E(G^*)=E(G)$, and each edge in $G^*$ corresponds to a common edge between a pair of faces of $G$.
	    
	    Let $T$ be a spanning tree of $G$. It is well-known that the subgraph $T^*$ of $G^*$ induced by the edge set $E(G^*)\setminus E(T)$ is a spanning tree of $G^*$.
	    
	    Let $\M$ and $\M^*$ be acyclic matchings on $\mathcal{F}(T)$ and $\mathcal{F}(T^*)$, respectively, such that all the edges of $T$ and $T^*$ are matched in $\M$ and $\M^*$. Proposition~\ref{tree} asserts the existence of such matchings. We may verify that $\M_0= \M \cup \{(x,y): x\in E(G), y \in F(G), (y,x)\in \M^*\}$ is an acyclic matching on $S^G$.  We observe that, all the elements of $S^G$ of rank $1$ are matched by $\M_0$. This implies, the face poset of $G$ is a sphere-like poset.
	    \end{proof}

\bibliographystyle{plainurl}
\bibliography{colorref.bib}	
	
\end{document}